\documentclass{amsart}
\usepackage[utf8]{inputenc}
\usepackage{amsmath}
\usepackage{amssymb}
\usepackage{amsthm}
\usepackage{bbm}
\usepackage{xcolor}
\usepackage{mathtools}
\usepackage{enumitem}
\usepackage{multirow}
\usepackage{caption}
\usepackage{subcaption}
\usepackage{wrapfig}

\newcommand{\vecsp}[2]{\mathbb{F}_{#1}^{#2}}

\newcommand{\st}{:\text{ }}

\newcommand{\bracketed}[1]{{(#1)}}
\newcommand{\vecline}[1]{\underline{#1}}

\newcommand{\rank}{\mathrm{rank}}

\newcommand{\vctwo}{\mathrm{VC}_2}
\def\VC{\mathrm{VC}}
\def\GS{\mathrm{GS}}
\def\F{\mathbb{F}}
\newcommand{\QGS}{\mathrm{QGS}}
\def\F{\mathbb{F}}
\def\fnz{\mathrm{fnz}}

\DeclarePairedDelimiter\gen{\langle}{\rangle}
\DeclarePairedDelimiter\abs{|}{|}

\DeclarePairedDelimiter\babs{\big|}{\big|}
\DeclarePairedDelimiter\set{\{}{\}}

\begin{document}

\newtheorem{theorem}{Theorem}[section]
\newtheorem{definition}[theorem]{Definition}
\newtheorem*{arl}{Arithmetic Regularity Lemma}
\newtheorem*{nntheorem}{Theorem}
\newtheorem{lemma}[theorem]{Lemma}
\newtheorem*{nclaim}{Claim}
\newtheorem{claim}[theorem]{Claim}
\newtheorem{corollary}[theorem]{Corollary}
\newtheorem{prop}[theorem]{Proposition}
\newtheorem{conj}[theorem]{Conjecture}

\theoremstyle{definition}
\newtheorem{example}[theorem]{Example}

\title{On a conjecture of Terry and Wolf}
\author{V. Gladkova}

\begin{abstract}
This paper shows that the $\VC_2$-dimension of a subset of $\F_p^n$ known as the `quadratic Green-Sanders example' is at least 3 and at most 501. Bounded $\VC_2$-dimension is a model-theoretic tameness property, whose effect on higher-order arithmetic regularity lemmas was recently studied by Terry and Wolf, as an analogue to bounded $\VC$-dimension in the linear setting. It is in that line of work that the quadratic Green-Sanders example was introduced, and the upper bound of 501 confirms a conjecture of Terry and Wolf. The techniques used in the proof are purely combinatorial and exploit a result in bipartite Ramsey theory. Additionally, the paper presents a much simplified proof of the fact that the (linear) Green-Sanders example has $\VC$-dimension at most 3.
\end{abstract}


\maketitle

\section{Introduction}

This paper concerns the so-called \emph{Green-Sanders examples}, a term coined by Terry and Wolf \cite{terry-wolf} to describe certain subsets of $\vecsp{p}{n}$ which are known to be counterexamples to some strengthenings of the arithmetic regularity lemma. An \emph{arithmetic regularity lemma} is a group-theoretic analogue of Szemer\'edi's regularity lemma for graphs \cite{szemeredi}: given a set $A \subseteq \F_p^n$, it provides a polynomially-structured partition of the space such that on most parts of the partition, $A$ is `uniform' in a certain sense. A `linear' version of the arithmetic regularity lemma, proved by Green \cite{tower-type}, states that for any $\epsilon > 0$ and any $A \subseteq \F_p^n$, there is a subspace $H \leqslant \vecsp{p}{n}$ such that $A$ is Fourier-uniform on all but an $\epsilon$-proportion of the cosets of $H$, with the codimension of $H$ depending only on the choice of $\epsilon$. Green \cite{tower-type} further showed that the codimension of $H$ can be as large as a tower of height $\Omega(\log_2 \epsilon^{-1})$, which was later strengthened to a tower of height $\epsilon^{-1}$ by Hosseini, Lovett, Moshkovitz and Shapira \cite{arl-lowerbound}.

This arithmetic regularity lemma cannot be improved to produce a subspace for which $A$ is Fourier-uniform on all of its cosets, as demonstrated by Green and Sanders \cite{green-sanders} via a construction which originated in the work of Bergelson, Deuber and Hindman \cite{bdh92}.

\begin{definition}[Linear Green-Sanders example]
    \label{def:linear-gs}
    Let $w_1, \ldots, w_n$ denote the standard basis of $\vecsp{p}{n}$. The \emph{linear Green-Sanders example} in $\vecsp{p}{n}$ is defined as
    $$\GS(p, n) = \set{x \in \vecsp{p}{n} \st \exists i \text{ } w_j^T x = 0 \text{ for }j < i, w_i^T x = 1}.$$
\end{definition}

However, Terry and Wolf \cite{terry-wolf-pre} (see also Conant \cite{gabespaper}) showed that such an improvement of the arithmetic regularity lemma can be obtained for a class of sets known as `stable' sets \cite[Definition 2]{terry-wolf-pre}. In fact, they proved that not only can one find a subspace $H$ such that a stable set $A$ is Fourier-uniform on all of its cosets, but in fact $A$ has near-trivial density (that is, almost 0 or almost 1) on all cosets; moreover, one can ensure that the codimension of $H$ is at most polynomial in $\epsilon^{-1}$ -- a drastic improvement on the tower-type bounds required for general sets. Alon, Fox and Zhao \cite{bounded-vc-arl} and Sisask \cite{sisask} subsequently generalised this to sets of bounded $\VC$-dimension, which is defined as follows in the context of abelian groups.

\begin{definition}[$\VC$-dimension of a subset]
\label{def:vcdim}
Given $k\geq 1$, a finite abelian group $G$ and a subset $A\subseteq G$, $A$ is said to have \emph{$\VC$-dimension at least $k$} if there exist elements $\{x_i: i\in [k]\}\cup \{y_S:S\subseteq [k]\}$ in $G$ such that $x_i+ y_S\in A$ if and only if $i\in S$.  The \emph{$\VC$-dimension of $A$} is the largest $k$ such that $A$ has $\VC$-dimension at least $k$.
\end{definition}

The arithmetic regularity lemma for sets of bounded $\VC$-dimension likewise achieves polynomial bounds on the codimension of $H$ as well as the near-trivial density of $A$ on most of the cosets, but does not guarantee it for all cosets. In this way, it can be thought of as a halfway point between the stable regularity lemma of Terry and Wolf \cite{terry-wolf-pre}, and the arithmetic regularity lemma of Green \cite{tower-type}.

\begin{theorem}[Arithmetic regularity lemma for sets of bounded $\mathrm{VC}$-dimension \cite{bounded-vc-arl, sisask}]
    \label{theorem:vc-arl}
    Fix $k \geq 1$ and $\epsilon > 0$. Suppose that $S \subseteq \vecsp{p}{n}$ has $\mathrm{VC}$-dimension of at most $k$. Then there is a subspace $H \leqslant \vecsp{p}{n}$ of index at most $\epsilon^{-k-o(1)}$ such that for all but an $\epsilon$-proportion of $c \in \vecsp{p}{n}$, either $\abs{S \cap (H+c)} \geq (1-\epsilon) \abs{H}$ or $\abs{S \cap (H+c)} \leq \epsilon \abs{H}$.
\end{theorem}

Indeed, exceptional cosets in Theorem \ref{theorem:vc-arl} are unavoidable since the linear Green-Sanders example has bounded $\mathrm{VC}$-dimension: specifically, Terry and Wolf showed that it has $\VC$-dimension at most $3$ \cite[Corollary A.10]{terry-wolf-v2}. Their argument for $p > 3$ was rather technically involved, and Section \ref{sec:vclin} of this paper presents a simpler proof, as well as establishing the following stronger theorem for $p \geq 5$.

\begin{theorem}
\label{thm:vcp5}
For all primes $p \geq 5$ and all $n\in \mathbb{N}$, the $\VC$-dimension of $\GS(p,n)$ is 2.
\end{theorem}

In \cite{terry-wolf} and \cite{terry-wolf-vc2dim}, Terry and Wolf undertook the study of strengthenings of the arithmetic regularity lemma in the quadratic setting. Here the place of stable sets is taken by so-called `$\mathrm{NFOP}_2$' sets \cite[Definition 1.14]{terry-wolf} and $\VC$-dimension is replaced by $\vctwo$-dimension, a generalisation of $\VC$-dimension introduced by Shelah in \cite{shelah} and defined by Terry and Wolf \cite[Definition 1.11]{terry-wolf} in the context of finite abelian groups as follows.

\begin{definition}[$\VC_2$-dimension of a subset]
\label{def:vc2dim}
Given $k\geq 1$, a finite abelian group $G$ and a subset $A\subseteq G$, $A$ is said to have \emph{$\VC_2$-dimension at least $k$} if there exist elements
$$
\{x_{i}:i\in [k]\}\cup \{y_i: i\in [k]\}\cup\{z_S: S\subseteq [k]^2\}  \subseteq G
$$
such that $x_i+ y_j+ z_S\in A$ if and only if $(i,j)\in S$. The \emph{$\VC_2$-dimension of $A$} is the largest $k$ such that $A$ has $\VC_2$-dimension at least $k$.
\end{definition}

Terry and Wolf established arithmetic regularity lemmas for sets of bounded $\vctwo$-dimension \cite[Theorem 1.12]{terry-wolf} and for $\mathrm{NFOP}_2$ sets \cite[Theorem 1.19]{terry-wolf}. They also showed that a quadratic generalisation of $\GS(p, n)$ does not satisfy the conclusions of the latter \cite[Theorem 1.21]{terry-wolf}.

\begin{definition}[Quadratic Green-Sanders example]
    \label{def:quadratic-gs}
    Let $M_1, \ldots, M_n$ be a basis for the space of symmetric $n \times n$ matrices of rank $n$ over $\vecsp{p}{}$, and let $Q_i$ denote the quadratic polynomial $Q_i(x) = x^T M_i x$. The \emph{quadratic Green-Sanders example} in $\mathbb{F}_p^n$ is defined as
    $$\QGS(p,n)=\set{x \in \vecsp{p}{n} \st \exists i \text{ } Q_1(x) = \ldots = Q_{i-1}(x) = 0, Q_i(x) = 1}.$$
\end{definition}
As observed in \cite{terry-wolf}, a suitable basis for the space of symmetric $n \times n$ matrices of rank $n$ exists for odd $n$ by \cite[Lemma 3]{high-rank-basis}.

\begin{figure}[t]
     \centering
     \hfill
     \begin{subfigure}[t]{0.48\textwidth}
         \centering
         \includegraphics[width=0.8\textwidth, alt={A sketch of the linear Green-Sanders example: a square is subdivided into three equal rectangles and the middle one is shaded and labelled $(1, *, *, *)$; the bottom rectangle is further subdivided into 3 equal rectangles, with the middle one shaded and labelled $(0, 1, *, *)$; and so on.}]{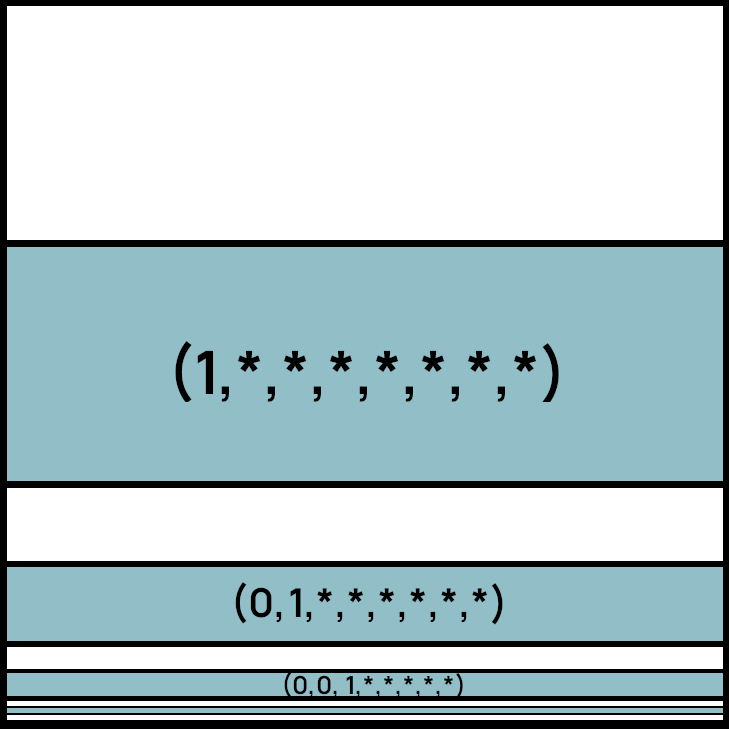}
         \caption{Linear Green-Sanders example}
     \end{subfigure}
     \hfill
     \begin{subfigure}[t]{0.48\textwidth}
         \centering
         \includegraphics[width=0.8\textwidth, alt={A sketch of the quadratic Green-Sanders example: exactly the same as the linear Green-Sanders, but rectangles are replaced with "wiggly" shapes, to convey that the level set is no longer a subspace.}]{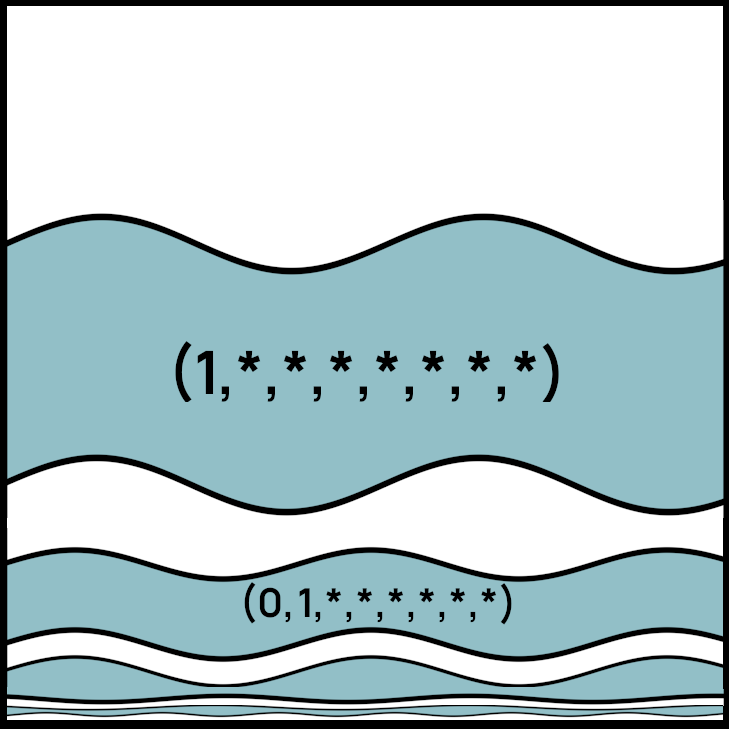}
         \caption{Quadratic Green-Sanders example}
     \end{subfigure}
     \hfill
     \caption{Green-Sanders examples in $\vecsp{3}{n}$. $\GS(3, n)$ and $\QGS(3, n)$, defined by level sets of linear and quadratic polynomials respectively.}
\end{figure}

Terry and Wolf further conjectured \cite[Conjecture 5.24]{terry-wolf-v2} that the quadratic Green-Sanders example has bounded $\VC_2$-dimension. If true, it would show that the arithmetic regularity lemma for sets of bounded $\vctwo$-dimension must be strictly weaker than that for $\mathrm{NFOP}_2$ sets. This paper confirms that this is indeed the case.

\begin{theorem}
    \label{theorem:upper}
   For all primes $p>2$ and all $n \in \mathbb{N}$, the $\VC_2$-dimension of $\QGS(p,n)$ is at most $\text{br}(5;3,3)$.
\end{theorem}

Here, $br(r;s,t)$ is the \emph{bipartite $r$-colour Ramsey number}, which is the least integer $N$ such that for any $r$-colouring of the edges of the complete bipartite graph $K_{N,N}$, there is a monochromatic copy of $K_{s,t}$. A result of Conlon \cite{bipartite-ramsey} implies the bound $br(r;3,3) \leq 4r^3+1$ (see Appendix \ref{appendix:bipartite-ramsey}), which in particular means that the $\VC_2$-dimension of $\QGS(p, n)$ is at most $501$.

Finally, Section \ref{sec:lower} exhibits a non-trivial lower bound on the $\VC_2$-dimension of the quadratic Green-Sanders example.

\begin{theorem}
\label{theorem:lower}
For all primes $p>2$ and all $n \geq 31$, the $\VC_2$-dimension of $\QGS(p,n)$ is at least 3.
\end{theorem}
\noindent It seems likely that the lower bound is closer to the truth, but the precise value of the $\vctwo$-dimension of $\QGS(p,n)$ remains to be determined.

\subsection*{Acknowledgments.} The author would like to thank Caroline Terry and Julia Wolf for introducing them to the subject and bringing this problem to their attention. Additional thanks to the anonymous reviewer for their careful reading of the initial manuscript and several helpful suggestions.

\section{The $\VC$-dimension of the linear Green-Sanders example}
\label{sec:vclin}

The concept of $\VC$-dimension may be more familiar to the reader in the following form, where $\mathcal{P}(V)$ denotes the power set of a given set $V$.

\begin{definition}[$\VC$-dimension of a family of sets]
\label{def:vcfam}
    Let $V$ be a set, and let $\mathcal{F}\subseteq \mathcal{P}(V)$. A set $S\subseteq V$ is said to be \emph{shattered by $\mathcal{F}$} if $\set{S \cap F \st F \in \mathcal{F}} = \mathcal{P}(S).$ The \emph{$\VC$-dimension of $\mathcal{F}$} is the largest integer $k$ (if it exists) such that $\mathcal{F}$ shatters some set of size $k$.
\end{definition}

It is easy to verify that this definition coincides with Definition \ref{def:vcdim} upon taking $V$ to be the finite abelian group $G$, and $\mathcal{F}=\{A-x: x\in G\}$ to be the family of translates of the set $A$. It will be convenient to bear this correspondence in mind as we proceed. In a slight abuse of language, $A$ will be said to \textit{shatter} a set $S$ when the family of its translates does.

It is not difficult to see that the set $\set{(0, 0, 0), (0, 1, -1), (0, -1, 1)}$ is shattered by translates of $\GS(3,n)$ \cite{terry-wolf}, and thus the $\VC$-dimension of $\GS(3,n)$ is at least 3. This section gives a short proof showing that the $\VC$-dimension of $\GS(p,n)$ is at most 3 for any $p \geq 3$, meaning that the $\VC$-dimension of $\GS(3,n)$ is equal to 3.
\begin{theorem}
\label{theorem:vc-gs}
For all primes $p > 2$ and all $n$, $\GS(p,n)$ has $\VC$-dimension at most 3.
\end{theorem}
There are two key observations that make the proof possible. Firstly, if a set $S$ is shattered by translates of $A$, then so is any translate of $S$. In particular, we may assume without loss of generality that $0 \in S$. The second observation is that any set of size 4 must have at least two non-zero elements that both lie in $A$ or both lie in $A^C$, by the pigeonhole principle. In Section \ref{sec:upper}, the pigeonhole principle will be replaced by a more sophisticated version of Ramsey's Theorem.

\begin{proof}[Proof of Theorem \ref{theorem:vc-gs}]
Let $\gen{x,y}$ denote the dot product $x^T y$, and let $\fnz: \vecsp{p}{n} \rightarrow [n+1]$ be the function indicating the position of the first non-zero coordinate of a vector, if it exists, defined by
$$\fnz(x) =
        \begin{cases}
            \min\set{k \st \gen{x,w_k} \neq 0} &\text{ if } x \neq \underline{0};\\
            n+1 &\text{ otherwise.}
        \end{cases}$$
Then $A$ can be rewritten as
$A=\set{x \in \vecsp{p}{n} \st \fnz(x) < n+1 \text{ and } \gen{x, w_{\fnz(x)}}=1}.$

Suppose there is a set $S \subseteq \vecsp{p}{n}$ of size at least 4 such that $S$ is shattered by translates of $A=\GS(p,n)$. Without loss of generality, we may assume that $0 \in S$. Moreover, since $\abs{S} \geq 4$, there must be two non-zero elements $a_1, a_2 \in S$ such that either $a_1, a_2 \in A$ or $a_1, a_2 \in A^C$.

\textbf{Case 1: $a_1, a_2 \in A$.} Since $A$ shatters $S$, there exists some $y \in \vecsp{p}{n}$ such that $y = 0+y \in A$ and $a_1+y, a_2+y \notin A$.
    
    Fix $j \in \set{1,2}$ and suppose that $\fnz(a_j) < \fnz(y)$. Then for all $i \leq \fnz(a_j)$,
    $$\gen{a_j+y,w_i}=\gen{a_j,w_i}+\gen{y,w_i} = \gen{a_j,w_i} \in \set{0, 1},$$
    since $a_j \in A$. However, this implies that $a_j+y \in A$, contradicting the choice of $y$. Similarly, $\fnz(y) < \fnz(a_j)$ is impossible, for then $\gen{a_j+y,w_i}=\gen{a_j,w_i}+\gen{y,w_i} = \gen{y,w_i}$ for all $i \leq \fnz(y)$, and thus $a_j+y \in A$ as before. Therefore $\fnz(y)=\fnz(a_j)$ for any $j \in \set{1,2}$, so in particular, $\fnz(a_1)=\fnz(a_2)$.

    There must also exist a different $y' \in \vecsp{p}{n}$ such that $y', a_1+y' \in A$ and $a_2+y' \notin A$. Arguing as in the preceding paragraph, it must be the case that $\fnz(y')=\fnz(a_2)$, which implies $\fnz(a_1)=\fnz(y')$. But then
    $$\gen{a_1+y',w_{\fnz(a_1)}} = \gen{a_1,w_{\fnz(a_1)}}+\gen{y',w_{\fnz(a_1)}} = 2 \notin \set{0,1}.$$
   Since $\fnz(a_1)=\fnz(y')$ implies $\fnz(a_1) \leq \fnz(a_1+y')$, the above equation shows that $a_1+y' \notin A$, which is again a contradiction.
   
\textbf{Case 2: $a_1, a_2 \notin A$.} Since $A$ shatters $S$, there exists some $y \in \vecsp{p}{n}$ such that $y \notin A$ and $a_1+y, a_2+y \in A$.
     
     Fix $j \in \set{1,2}$. If $\fnz(a_j) > \fnz(y)$, then $\gen{a_j+y,w_{\fnz(a_j+y)}}=\gen{y,w_{\fnz(y)}} \notin \set{0,1}$, contradicting the choice of $y$. Similarly, we cannot have $\fnz(y) > \fnz(a_j)$, as that would imply $\gen{a_j+y,w_{\fnz(a_j+y)}}=\gen{a_j,w_{\fnz(a_j)}} = 1$. Therefore, $\fnz(y)=\fnz(a_j)$ for any $j \in \set{1,2}$, so in particular, $\fnz(a_1)=\fnz(a_2)$.
    
    There must further exist a $y' \in \vecsp{p}{n}$ such that $y', a_1+y' \in A$ and $a_2+y' \notin A$. Since $\fnz(a_2) > \fnz(y')$ would imply $\gen{a_2+y',w_{\fnz(a_2+y')}}= \gen{y',w_{\fnz(y')}} = 1$, it must be the case that $\fnz(y') \geq \fnz(a_2) = \fnz(a_1)$. Similarly, $\fnz(y') > \fnz(a_1)$ would mean that $\gen{a_1+y',w_{\fnz(a_1+y')}}=\gen{a_1,w_{\fnz(a_1)}} \notin \set{0, 1}$, so we must have $\fnz(y') = \fnz(a_1)$, and therefore $\fnz(a_1+y') \geq \fnz(a_1)$.
    
    However, $\fnz(a_1+y')=\fnz(a_1)$ would imply
    $1=\gen{a_1+y',w_{\fnz(a_1+y')}}=\gen{a_1,w_{\fnz(a_1)}}+1 \neq 1,$
    which forces us to conclude that $\fnz(a_1+y') > \fnz(a_1)$. As a consequence, $\gen{a_1,w_{\fnz(a_1)}} = -\gen{y',w_{\fnz(y')}} = -1.$ By a symmetric argument, choosing $y'$ to be so that $y', a_2+y' \in A$ and $a_1+y' \notin A$, we must have $\gen{a_2,w_{\fnz(a_2)}}=-1$ also.
    
    Finally, there exists $y'' \in \vecsp{p}{n}$ such that $y'', a_1+y'' \notin A$ and $a_2+y'' \in A$. It is once again the case that $\fnz(a_2) = \fnz(y'') = \fnz(a_1)$ and therefore
    $$\gen{a_1+y'',w_{\fnz(y'')}}=\gen{a_2+y'',w_{\fnz(y'')}}=-1+\gen{y'',w_{\fnz(y'')}} \neq 0.$$
    However, this is impossible, as $\fnz(a_1+y''), \fnz(a_2+y'') \geq \fnz(y'')$ and the first non-zero coordinates of $a_1+y''$ and $a_2+y''$ must be different.
\vspace{0.1cm}

As both cases lead to a contradiction, $S$ cannot have size greater than 3.
\end{proof}

When $p \geq 5$, it turns out that a contradiction can be forced even when $a_1 \in A, a_2 \notin A$ which rules out the possibility of $\abs{S} \geq 3$. This results in Theorem \ref{thm:vcp5}.

\begin{proof}[Proof of Theorem \ref{thm:vcp5}] Let $p \geq 5$, and suppose that $S$ is a set of size 3 shattered by $A$. Without loss of generality, it may be assumed that $S = \set{0, a_1, a_2}$ for some non-zero $a_1, a_2 \in \vecsp{p}{n}$. Moreover, we must have $a_1 \in A, a_2 \notin A$ as a consequence of the proof above.

Since $A$ shatters $S$, there exists some $y \in \vecsp{p}{n}$ such that $y, a_1+y \notin A$ and $a_2+y \in A$. Each of $\fnz(a_2) < \fnz(y)$ and $\fnz(y) < \fnz(a_2)$ would imply $\gen{a_2+y, w_{\fnz(a_2+y)}} \neq 1$, so we must have $\fnz(a_2)=\fnz(y)$. On the other hand, if $\fnz(a_1) < \fnz(y)$, then $$\gen{a_1+y,w_{\fnz(a_1+y)}} = \gen{a_1,w_{\fnz(a_1)}} = 1,$$ contradicting the choice of $y$, so $\fnz(a_1) \geq \fnz(y) = \fnz(a_2)$.

There also exists some $y' \in \vecsp{p}{n}$ such that $y', a_2+y' \in A$ and $a_1+y' \notin A$. Similarly to the previous paragraph, this yields the conclusion $\fnz(a_2) \geq \fnz(y') = \fnz(a_1)$. Therefore, $\fnz(a_1)=\fnz(a_2)$ and, in particular, we must have
$$\gen{a_2,w_{\fnz(a_2)}}+\gen{y',w_{\fnz(y')}} = \gen{a_2+y',w_{\fnz(y')}} \in \set{0,1},$$
which implies $\gen{a_2,w_{\fnz(a_2)}} = -1$.

Finally, there exists a $y'' \in \vecsp{p}{n}$ such that $y'' \notin A$ and $a_1+y'', a_2+y'' \in A$. As in the case of $y$, this leads to $\fnz(a_2) = \fnz(y'')$, which in turn means that $\fnz(a_1) = \fnz(y'')$. Then $\fnz(a_j) \leq \fnz(a_j+y'')$ for each $j \in \{1,2\}$, and hence
$$\gen{a_j,w_{\fnz(a_j)}}+\gen{y'',w_{\fnz(y'')}} = \gen{a_j+y'',w_{\fnz(a_j)}} \in \set{0,1}.$$
However, this only holds if $\gen{y'', w_{\fnz(y'')}} = -1$ and $\gen{a_2,w_{\fnz(a_2)}} = 2$, a contradiction.

To complete the proof, observe that $A$ does indeed shatter a set of size 2: for example, $S = \set{0, w_2}$ can be shattered by taking $y$ from the set $\set{0, w_1, -w_1, w_2}$.
\end{proof}

\section{The quadratic Green-Sanders example has bounded $\VC_2$-dimension}
\label{sec:upper}

For convenience, the definition of $\vctwo$-dimension is restated below in terms of a function that assigns to each $(i,j) \in [0,k]^2$ a value from $\set{A, A^C}$. Here $[0,k]$ denotes $\{0,1,2,\dots,k\}$, as opposed to the standard $[k]=\{1,2,\dots,k\}$.

\begin{definition}[Containment map]\label{def:contmap}
    For a set $A$ and an integer $k \geq 1$, a function $\phi:[0, k]^2 \rightarrow \set{A, A^C}$ is called a \emph{containment map on $[0,k]^2$}. If $\phi$ is defined only on a subset of $[0, k]^2$, then $\phi$ is a \emph{partial containment map on $[0,k]^2$}.
\end{definition}

\begin{definition}[$\VC_2$-dimension, containment map form]
Given an integer $k\geq 1$, a finite abelian group $G$ and a subset $A\subseteq G$, $A$ is said to have \emph{$\VC_2$-dimension at least $k$} if there exist sets $X = \set{x_0, \ldots, x_{k-1}}$ and $Y = \set{y_0, \ldots, y_{k-1}}$ in $G$ such that the following holds.
For any containment map $\phi$ on $[0, k-1]^2$, there exists $z \in G$ such that $x_i+y_j+z \in \phi(i, j)$. If this is the case, we say that \emph{$z$ realises $\phi$}.
  
The tuple $(X,Y)$ is said to be \emph{quadratically shattered by $A$} if every containment map on $[0, k-1]^2$ can be realised.
\end{definition}

As in the linear case, it may be assumed without loss of generality that $x_0 = y_0 = 0$: if $(X, Y)$ is quadratically shattered by $A$, then so is $(X+t_X, Y+t_Y)$ for any $t_X, t_Y \in \vecsp{p}{n}$, as one need only replace $z$ by $z - t_X - t_Y$ to satisfy any containment map.

Throughout this section and the rest of the paper, let $A=\QGS(p,n)$, let $M_1,\dots, M_n$ be the matrices defining $A$ as in Definition \ref{def:quadratic-gs}, and let $Q_1, \ldots, Q_n: \vecsp{p}{n} \rightarrow \vecsp{p}{}$ be quadratic polynomials defined by $Q_t(x) = x^T M_t x$. The following (easily verified) identity will be used repeatedly. For any $x, y, z \in \vecsp{p}{n}$, and any $t\in[n]$,
    \begin{equation}
        \label{eq:expand}
        Q_t(x+y+z) = Q_t(x+z) + Q_t(y+z) - Q_t(z) + 2 x^T M_t y.
    \end{equation}

In the previous section, the proof of Theorem \ref{theorem:vc-gs} proceeded by finding an unavoidable `monochromatic' structure (two elements that are both in $\GS(p,n)$ or both in $\GS(p,n)^C$), then varying the (linear) containment map to deduce information about the elements of a set $S$, which was assumed to be shattered by $\GS(p,n)$. The proof of Theorem \ref{theorem:upper} follows the same general strategy but utilises a larger `monochromatic' structure at the start. To deduce information about a tuple of sets $(X,Y)$ that is assumed to be shattered by $A$, we will use equations of the form \eqref{eq:expand}. Such deductions are possible since, while $z$ will vary depending on the choice of the containment map, the values of $x$, $y$, and the corresponding cross-term $x^T M_t y$ will remain constant.

\begin{example}
    Let $A' = \set{x \in \vecsp{p}{n} \st Q_1(x) = 1}$, which is a subset of $A = \QGS(p, n)$. Suppose that $\vctwo(A') \geq 2$. Then there exist sets $X = \set{0, x_1}$ and $Y = \set{0, y_1}$ such that $(X, Y)$ is quadratically shattered by $A'$. Let $\phi_1$ and $\phi_2$ be the two containment maps on $[0,1]^2$ satisfying $\phi_i(0,0) = \phi_i(0,1) = \phi_i(1,0) = A'$. For any $z_i$ realising $\phi_i$, we have $Q_1(x_1+y_1+z_i) = 1 + 2x_1^T M_1 y_1$ by \eqref{eq:expand}. In particular, the value of $Q_1(x_1+y_1+z_i)$ is constant, regardless of $z_i$. However, $\phi_1(1,1) \neq \phi_2(1,1)$, so one of $\phi_1$ or $\phi_2$ cannot be realised. Therefore, the initial assumption was false and $\vctwo(A') = 1$.
\end{example}

The particular `monochromatic' structure the proof requires is a $3 \times 3$ subgrid in $[k]^2$ such that all cross-terms $x_i^T M_t y_j$ (corresponding to $(i,j)$ in the subgrid) have the same value for some $t \in [n]$. It turns out that if $(X, Y)$ is quadratically shattered by $A$, then under certain conditions, this value \textit{must} be 0, as the following proposition demonstrates.
\begin{prop}
\label{prop:3grid}
   Fix integers $k \geq 3$ and $m \in [n]$, and let $X = \set{0, x_1, \ldots, x_{k}}$, ${\color{black}Y = \set{0, y_1, \ldots, y_{k}}}$ be subsets of $\vecsp{p}{n}$. Suppose that for all $t < m$ and $(i,j) \in [0, k]^2$, the cross-term $\mu_{i, j}^\bracketed{t} = 2 x_i^T M_t y_j$ is equal to zero. Then there exists a containment map $\phi$ on $[0, 3]^2$ such that for any $z$ realising $\phi$ the following holds.
    \begin{enumerate}[label =\normalfont (\roman*)]
        \item \label{prop:gridlemma1} For any $t < m$ and $(i,j) \in [3]^2$,
        $Q_t(z) = Q_t(x_i+z) = Q_t(y_j+z) = 0.$

        \item \label{prop:gridlemma2} If $\mu_{i, j}^\bracketed{m}$ is constant for all $(i,j) \in [3]^2$, then \ref{prop:gridlemma1} holds for $t=m$, and
        $\mu_{i, j}^\bracketed{m} = 0.$
    \end{enumerate}
    Moreover, the value of $\phi(0,0)$ may be assigned arbitrarily.
\end{prop}
    \begin{proof}
    Define $\phi$ to be equal to $A^C$ at $(0,1), (0,2), (1,0), (2,0), (2,3),$ and $(3,2)$, and $A$ everywhere else, with $\phi(0, 0)$ set arbitrarily and $x_0 = y_0 = 0$. Then by \eqref{eq:expand},
    \begin{equation}
        \label{eq:expand1}
        Q_t(x_i+y_j+z) = Q_t(x_i+y_0+z) + Q_t(x_0+y_j+z) - Q_t(z) + \mu_{i, j}^\bracketed{t}.
    \end{equation}
    Fix $z \in \vecsp{p}{n}$ such that $z$ realises $\phi$ (if it exists), and write $\Lambda_t(z) = Q_t(z) - \mu_{i, j}^\bracketed{t}$. We will first prove, by induction on $t$, that for all $t \leq m$ and all $(i, j) \in [3]^2$,
        \begin{equation}
            \label{eq:lemma-goal}
            \Lambda_t(z) = Q_{t}(x_i+z) = Q_{t}(y_j+z) = 0.
        \end{equation}
    This, in particular, will imply that when $t < m$, $Q_t(z) = \Lambda_t(z)+\mu_{i,j}^\bracketed{t} = \Lambda_t(z) = 0$, which proves part \ref{prop:gridlemma1}.

    Suppose, as the inductive hypothesis, that \eqref{eq:lemma-goal} holds for all $t' < t$ (the base case at $t=1$ is trivial). Then for all such $t'$ and all $(i,j) \in [0, 3]^2$, $Q_{t'}(x_i+y_j+z) = 0$ by \eqref{eq:expand1}. As a result, the value of $Q_{t'}$ cannot determine whether or not $x_i+y_j+z$ is contained in $A$, so that task is passed on to $Q_t$. In particular, this means that the possible values of $Q_t$ at $x_i+y_j+z$ are restricted by $\phi(i,j)$: for instance, if $\phi(i,j)=A$, then $Q_t(x_i+y_j+z)$ must be either $0$ or $1$.
    
    Using a superscript to note such restrictions on the value of $Q_{t}(x_i+y_j+z)$, we have the following as a consequence of equation \eqref{eq:expand1} and the particular choice of $\phi$:
    \begin{align}
        \label{eq:11} (1,1):\hspace{0.5cm} &Q_{t}^{{\color{black} \neq 1}}(x_1+y_0+z) + Q_{t}^{{\color{black} \neq 1}}(x_0+y_1+z) - \Lambda_t(z) \in \set{0,1}\\
        \label{eq:12} (1,2):\hspace{0.5cm} &Q_{t}^{{\color{black} \neq 1}}(x_1+y_0+z) + Q_{t}^{{\color{black} \neq 1}}(x_0+y_2+z) - \Lambda_t(z) \in \set{0,1}\\
        \label{eq:13} (1,3):\hspace{0.5cm} &Q_{t}^{{\color{black} \neq 1}}(x_1+y_0+z) + Q_{t}^{{\color{black} = 0, 1}}(x_0+y_3+z) - \Lambda_t(z) \in \set{0,1}\\
        (\label{eq:21} 2,1):\hspace{0.5cm} &Q_{t}^{{\color{black} \neq 1}}(x_2+y_0+z) + Q_{t}^{{\color{black} \neq 1}}(x_0+y_1+z) - \Lambda_t(z) \in \set{0,1}\\
        \label{eq:22} (2,2):\hspace{0.5cm} &Q_{t}^{{\color{black} \neq 1}}(x_2+y_0+z) + Q_{t}^{{\color{black} \neq 1}}(x_0+y_2+z) - \Lambda_t(z) \in \set{0,1}\\
        \label{eq:23} (2,3):\hspace{0.5cm} &Q_{t}^{{\color{black} \neq 1}}(x_2+y_0+z) + Q_{t}^{{\color{black} = 0, 1}}(x_0+y_3+z) - \Lambda_t(z) \neq 1\\
        \label{eq:31} (3,1):\hspace{0.5cm} &Q_{t}^{{\color{black} = 0, 1}}(x_3+y_0+z) + Q_{t}^{{\color{black} \neq 1}}(x_0+y_1+z) - \Lambda_t(z) \in \set{0,1}\\
        \label{eq:32} (3,2):\hspace{0.5cm} &Q_{t}^{{\color{black} = 0, 1}}(x_3+y_0+z) + Q_{t}^{{\color{black} \neq 1}}(x_0+y_2+z) - \Lambda_t(z) \neq 1\\
        \label{eq:33} (3,3):\hspace{0.5cm} &Q_{t}^{{\color{black} = 0, 1}}(x_3+y_0+z) + Q_{t}^{{\color{black} = 0, 1}}(x_0+y_3+z) - \Lambda_t(z) \in \set{0,1}
    \end{align}

    \noindent  Firstly, note that if $Q_{t}(x_1+y_0+z) - \Lambda_t(z) \neq 0$, then it must be equal to $\pm 1$ by \eqref{eq:13}.
    
    \textbf{Case 1. $Q_{t}(x_1+y_0+z) - \Lambda_t(z) = -1$.} Then $Q_{t}(x_0+y_3+z) = 1$ and, taking equations \eqref{eq:11} and \eqref{eq:12} into account, $Q_{t}(x_0+y_1+z) = Q_{t}(x_0+y_2+z) = 2.$
    In turn, equations \eqref{eq:31} and \eqref{eq:32} imply that $Q_{t}(x_3+y_0+z) - \Lambda_t(z) = -2,$
    which contradicts \eqref{eq:33}.
    
    \textbf{Case 2. $Q_{t}(x_1+y_0+z) - \Lambda_t(z) = 1$.} Then $Q_{t}(x_0+y_3+z) = 0$ and, due to \eqref{eq:11} and \eqref{eq:12},
    $Q_{t}(x_0+y_1+z)$ and $Q_{t}(x_0+y_2+z)$ can take values of either $0$ or $-1$. This results in three subcases.

    \textbf{Subcase 2.1.} $Q_{t}(x_0+y_1+z) = -1.$
        Equation \eqref{eq:23} implies that $Q_{t}(x_2+y_0+z) - \Lambda_t(z) \neq 1$, which means that in \eqref{eq:21}, we must have $Q_{t}(x_2+y_0+z) - \Lambda_t(z) = 2$. As a result, \eqref{eq:22} gives $Q_t(x_0+y_2+z) = -1 = Q_{t}(x_0+y_1+z)$.

        Now consider \eqref{eq:31} and \eqref{eq:32}. It must be the case that $Q_{t}(x_3+y_0+z) - \Lambda_t(z) = 1$, so
        \begin{equation}
            \label{eq:subcase(a)}
            Q_{t}^{{\color{black} = 0,1}}(x_3+y_0+z) - \Lambda_t(z) = 1 = Q_t^{{\color{black} \neq 1}}(x_1+y_0+z) - \Lambda_t(z).
        \end{equation} But the only way we can have $Q_{t}^{{\color{black} = 0,1}}(x_3+y_0+z) = Q_t^{{\color{black} \neq 1}}(x_1+y_0+z)$ is if they are both zero, meaning that $\Lambda_t(z)$ must be equal to $-1$. Substituting this into $Q_{t}(x_2+y_0+z) - \Lambda_t(z) = 2$ gives $Q_{t}^{{\color{black} \neq 1}}(x_2+y_0+z) = 1$, a contradiction.

        \textbf{Subcase 2.2.} $Q_{t}(x_0+y_2+z) = -1.$ This can be handled in exactly the same way as Subcase 2.1, by noting that equations \eqref{eq:11}-\eqref{eq:23} are symmetric in $y_1$ and $y_2$, resulting in $Q_{t}(x_2+y_0+z) - \Lambda_t(z) = 2$ and $Q_t(x_0+y_1+z) = -1$, and \eqref{eq:subcase(a)} follows similarly.

        \textbf{Subcase 2.3.} $Q_{t}(x_0+y_1+z) = Q_{t}(x_0+y_2+z) = 0.$ Since $Q_{t}(x_2+y_0+z) - \Lambda_t(z) \neq 1$ by equation \eqref{eq:23}, equation \eqref{eq:21} implies that $Q_{t}(x_2+y_0+z) - \Lambda_t(z) = 0$. At the same time, equations \eqref{eq:31} and \eqref{eq:32} give $Q_{t}(x_3+y_0+z) - \Lambda_t(z) = 0$, so
        \begin{equation*}
            Q_{t}^{{\color{black} \neq 1}}(x_2+y_0+z) = \Lambda_t(z) = Q_{t}^{{\color{black} = 0,1}}(x_3+y_0+z).
        \end{equation*}
        This can only hold if $\Lambda_t(z) = 0$ which means that $Q_{t}^{{\color{black} \neq 1}}(x_1+y_0+z) = 1$, a contradiction.
        \vspace{0.2cm}

        Hence, neither of the two cases are possible, so we must have $Q_{t}(x_1+y_0+z) - \Lambda_t(z) = 0$. Substituting this into equations \eqref{eq:11} and \eqref{eq:12} gives  $Q_{t}(x_0+y_1+z) = Q_{t}(x_0+y_2+z) = 0$. Then \eqref{eq:31} and \eqref{eq:32} imply $Q_{t}(x_3+y_0+z) - \Lambda_t(z) = 0$, and therefore
        \begin{equation*}
            Q_{t}^{{\color{black} \neq 1}}(x_1+y_0+z) = \Lambda_t(z) = Q_{t}^{{\color{black} = 0,1}}(x_3+y_0+z),
        \end{equation*}
        which is only possible if $\Lambda_t(z) = 0$. With the deductions already made, equation \eqref{eq:21} now says that $Q_{t}^{{\color{black} \neq 1}}(x_2+y_0+z) \in \set{0,1}$, meaning that $Q_{t}(x_2+y_0+z) = 0$. Finally, \eqref{eq:23} results in $Q_{t}(x_0+y_3+z) = 0$. In other words, \eqref{eq:lemma-goal} holds for $t' = t$, which completes the induction and thus the proof of part \ref{prop:gridlemma1}.

        To prove \ref{prop:gridlemma2}, suppose that $\mu_{i,j}^\bracketed{m} \neq 0$. As $\phi(0,0)$ was set arbitrarily, \eqref{eq:lemma-goal} must hold at $t=m$ regardless of the choice of $\phi(0,0)$. However, setting $\phi(0,0) = A$ if $\mu_{i,j}^\bracketed{m} \neq 1$, and $\phi(0,0) = A^C$ otherwise results in $\Lambda_m(z) = Q_{m}(z) - \mu_{i,j}^\bracketed{m} \neq 0$, which contradicts \eqref{eq:lemma-goal}. Therefore $\mu_{i,j}^\bracketed{m} = Q_m(z) = 0$ for any choice of $\phi(0,0)$.
    \end{proof}

Note that when $m=1$ in Proposition \ref{prop:3grid}, there are no starting conditions on $\mu_{i,j}^\bracketed{t}$ since there is no $t < m$. The proposition then says that if there is a $3 \times 3$ subgrid of $[k]^2$ on which all $\mu_{i,j}^\bracketed{1}$ have the same value, then this value must be 0. It turns out that once we know this, we can deduce that $\mu_{i,j}^\bracketed{1}$ must be 0 for all $(i,j) \in [0, k]^2$. Of course, Proposition \ref{prop:3grid} may then be applied with $m=2$ and so on, until we deduce that $\mu_{i,j}^\bracketed{t} = 0$ for all $t$ and $(i,j)$.

A simple but useful observation can reduce the final bound on the $\vctwo$-dimension of $A$ substantially. In order to ensure the existence of a suitable $3 \times 3$ subgrid at level $m$, we will apply a bipartite Ramsey theorem (see Proposition \ref{prop:Ramseybound} in Appendix \ref{appendix:bipartite-ramsey}) to the colouring given by the values of $\mu_{i,j}^\bracketed{m}$ on $(i,j)$. While ordinarily this colouring may consist of as many as $p$ colours, it turns out that in the particular setting we are interested in, five colours are always sufficient.

\begin{lemma}
    \label{lemma:5colour}
    Let the assumptions be as in Proposition \ref{prop:3grid}. If $(X,Y)$ is quadratically shattered by $A$, then for any pair $(i, j) \in [k]^2$, $\mu_{i,j}^\bracketed{m} \in \set{-2, -1, 0, 1, 2}$.
\end{lemma}

\begin{proof}
    Fix $(i, j) \in [k]^2$ which, by relabelling, may be assumed to be equal to $(3, 3)$. Let $\phi$ be the containment map on $[0, 3]^2$ given by Proposition \ref{prop:3grid} with $\phi(0,0) = A$. Since $(X,Y)$ is quadratically shattered by $A$, there exists some $z \in \vecsp{p}{n}$ realising $\phi$. Then for all $t < m$, $Q_t(z) = Q_t(x_3+z) = Q_t(y_3+z) = Q_t(x_3+y_3+z) = 0$ by Proposition \ref{prop:3grid}\ref{prop:gridlemma1}. Moreover, the definition of $\phi$ in the proof of Proposition \ref{prop:3grid} means that all of $z, x_3+z, y_3+z$, and $x_3+y_3+z$ are contained in $A$. Then
    \begin{equation*}
        Q_m^{{\color{black} = 0, 1}}(x_3+z) + Q_m^{{\color{black} = 0, 1}}(y_3+z) - Q_m^{{\color{black} = 0, 1}}(z) + \mu^\bracketed{m}_{3,3} \in \set{0,1},
    \end{equation*}
    and hence $\mu^\bracketed{m}_{3,3} \in \set{0,1} - \set{0,1} - \set{0,1} + \set{0,1} = \set{-2, -1, 0, 1, 2}.$
\end{proof}

\begin{proof}[Proof of Theorem \ref{theorem:upper}]
    Let $k = br(5; 3, 3)$, and consider sets $X = \set{x_0, x_1, \ldots, x_k}$ and $Y = \set{y_0, y_1, \ldots, y_k}$, with $x_0 = y_0 = 0$. Suppose that $(X,Y)$ is quadratically shattered by $A$. We will prove by induction that for all $(i, j) \in [k]^2$ and all $t \in [n]$, the cross-term $\mu_{i,j}^\bracketed{t} = 2 x_i^T M_t y_j$ is equal to $0$.

    For the inductive step, assume that the desired conclusion holds for all $t < m$, where $m \leq n$. The goal is to show that this property can be extended to $t = m$. Note that the statement is trivially true for $m=1$, which serves as the base case.
    
    The inductive hypothesis implies that the assumptions of Lemma \ref{lemma:5colour} are satisfied, and therefore each $\mu_{i,j}^\bracketed{m}$ takes one of at most five possible values. If $(i,j) \in [k]^2$ is assigned the value of $\mu_{i,j}^\bracketed{m}$ as a colour, then by the choice of $k$, there is a monochromatic $3 \times 3$ subgrid. Without loss of generality, assume this subgrid is $[3]\times [3]$, and let $\lambda_m$ denote its colour. Proposition \ref{prop:3grid} now gives a containment map $\phi$ on $[0, 3]^2$ such that, for all $t \leq m$, any $z$ realising $\phi$ must satisfy $\lambda_t = Q_t(z) = Q_{t}(x_i+y_0+z) = Q_t(x_0+y_i+z) = 0,$
    and this still holds if the value of $\phi(0,0)$ is changed. The following claim shows that the entire grid must have the same colour as the chosen $3 \times 3$ subgrid, namely $\lambda_m=0$.
    \begin{nclaim}
        For all $x \in X$ and $y \in Y$, $2 x^T M_m y = 0$.
    \end{nclaim}
    \begin{proof}[Proof of Claim]
    \renewcommand{\qedsymbol}{}
        Note that $x_i^T M_m 0 = 0^T M_m y_j = 0$ for any $(i, j) \in [k]^2$, so the whole of the $4 \times 4$ subgrid $\set{0, x_1, \ldots, x_3} \times \set{0, y_1, \ldots, y_3}$ has colour 0. Taking this as the base case, the subgrid can be grown by induction. Specifically, assuming that the grid $\set{0, x_1, \ldots, x_{i-1}} \times \set{0, y_1, \ldots, y_{i-1}}$ is monochromatic with colour $0$, we will show that so is $\set{0, x_1, \ldots, x_{i}} \times \set{0, y_1, \ldots, y_{i}}$.
        
        Fix some $j < i$ and suppose that $\mu^\bracketed{m}_{j,i} \neq 0$. Relabelling if necessary, assume that $j \neq 1, 2$. By inductive hypothesis, the subgrid $\set{x_1, x_2, x_j} \times \set{y_1, y_2, y_3}$ is monochromatic and, moreover, the other conditions of Proposition \ref{prop:3grid} are satisfied. This means that there is a containment map $\phi'$ on $\set{0, 1, 2, j} \times [0,3]$ such that any $z$ realising it satisfies
        $Q_t(x_1+y_0+z) = Q_t(x_2+y_0+z) = Q_t(x_j+y_0+z) = Q_t(z) = 0.$
        
        As a result, equation \eqref{eq:expand} yields the following for all $t \leq m$:
        \begin{align}
            \label{eq:1i} Q_t(x_1+y_i+z) &= Q_t(x_0+y_i+z) + \mu^\bracketed{t}_{1,i};\\
            \label{eq:2i} Q_t(x_2+y_i+z) &= Q_t(x_0+y_i+z) + \mu^\bracketed{t}_{2,i};\\
            \label{eq:ji} Q_t(x_j+y_i+z) &= Q_t(x_0+y_i+z) + \mu^\bracketed{t}_{j,i}.
        \end{align}
        Since the values of $\phi'$ are only fixed on $\set{0, 1, 2, j} \times [0,3]$, $\phi'$ can be extended to the rest of the grid in any manner without losing \eqref{eq:1i}-\eqref{eq:ji}.
        
        First, set $\phi'(1, i)$ and $\phi'(2, i)$ to have different values. For all $t < m$, this immediately implies that $Q_t(x_0+y_i+z) = 0$ since $\mu^\bracketed{t}_{1,i} = \mu^\bracketed{t}_{2,i} = 0$ by the inductive hypothesis.
        
        Next, set $\phi'(j, i) = A^C$ if $\mu^\bracketed{m}_{j,i} = 1$, and $\phi'(j, i) = A$ otherwise. Additionally, let $\phi'(0, i) = A$. With these choices, equations \eqref{eq:1i}-\eqref{eq:ji} mean that we cannot have $Q_{m}(x_0+y_i+z) = 0$, so it must be equal to $1$. In addition, if $\phi'(j, i) = A$, then $\mu^\bracketed{m}_{j,i} = -1$, so the only possible values $\mu^\bracketed{m}_{j,i}$ can take are $\pm 1$.

        The only condition made on $\phi'(1, i)$ and $\phi'(2, i)$ so far is that they take different values. Then one of them, e.g.~$\phi'(1,i)$ must be equal to $A^C$, which implies $\mu^\bracketed{m}_{1,i} \neq 0$ by \eqref{eq:1i}. On the other hand, swapping the assignments of $(1,i)$ and $(2,i)$ changes none of the previous deductions, so it follows that also $\mu^\bracketed{m}_{2,i} \neq 0$. Moreover, arguing as in the preceding paragraph, both $\mu^\bracketed{m}_{1,i}$ and $\mu^\bracketed{m}_{2,i}$ can only take values of $\pm 1$.

        \textbf{Case 1: $\mu^\bracketed{m}_{ji} = 1$.} Redefine $\phi'(j, i) = A$ while keeping $\phi'(0, i) = A$. This results in $Q_{m}(x_0+y_i+z) = 0$ by \eqref{eq:ji} and, consequently, $Q_{m}(x_1+y_i+z) = \mu^\bracketed{m}_{1,i}$ by \eqref{eq:1i}. Setting $\phi(1, i)$ so that it must differ from $\mu^\bracketed{m}_{1,i}$ gives a contradiction.

        \textbf{Case 2: $\mu^\bracketed{m}_{ji} = -1$.} An argument symmetric to Case 1 shows that $\mu^\bracketed{m}_{1,i} = \mu^\bracketed{m}_{2,i} = - 1$ also. Redefine $\phi'(j, i) = A$ and $\phi'(0, i) = A^C$. This implies $Q_{m}(x_0+y_i+z) = 2$ by \eqref{eq:ji}, but then all three equations \eqref{eq:1i}-\eqref{eq:ji} are equal to 1 and, therefore, correspond to the same containment value, in contradiction with $\phi'(1, i) \neq \phi'(2, i)$.

        With both these options eliminated, the only possible conclusion is that $\mu^\bracketed{m}_{j,i}$ must have been $0$ in the first place, and hence the subgrid $\set{0, x_1, \ldots, x_{i-1}} \times \set{0, y_1, \ldots, y_{i}}$ is monochromatic. By symmetry, $\set{0, x_1, \ldots, x_{i}} \times \set{0, y_1, \ldots, y_{i-1}}$ is likewise monochromatic. It remains to show that $\mu_{i,i}^\bracketed{m} = 0$.

        By Proposition \ref{prop:3grid}, there exists a containment map $\phi''$ on $[0,3] \times \set{0, 1, 2, i}$ such that $Q_{t}(x_0+y_1+z) = Q_{t}(x_0+y_2+z) = Q_t(x_0+y_i+z) = Q_t(z) = 0$ for any $z$ realising $\phi''$ and all $t \leq m$. With the information already deduced, this implies
        \begin{align}
            \label{eq:i1} Q_{t}(x_i+y_1+z) &= Q_{t}(x_i+y_0+z)\\
            \label{eq:i2} Q_{t}(x_i+y_2+z) &= Q_{t}(x_i+y_0+z)\\
            \label{eq:ii} Q_{t}(x_i+y_i+z) &= Q_{t}(x_i+y_0+z) + \mu^\bracketed{t}_{i,i}.
        \end{align}
        Setting $\phi''(i,1) \neq \phi''(i,2)$ forces $Q_{t}(x_i+y_0+z) = 0$ for all $t \leq m$. However, since \eqref{eq:ii} must be satisfied for any choice of $\phi''(i, i)$, the only conclusion is that $\mu_{i,i}^\bracketed{m} = 0$, which completes the proof of the claim.
    \end{proof}
     It follows, by induction, that if $(X,Y)$ is quadratically shattered by $A$, then $\mu^\bracketed{t}_{j,i} = 0$ for all $t \leq n$ and $(i, j) \in [0,k]^2$. By one final application of Proposition \ref{prop:3grid}, there is a containment map $\phi$ on $[0,3]^2$ such that $Q_t(z) = 0$ for all $t \leq n$, and this holds even if the value of $\phi(0, 0)$ is changed. However, no $z$ with this property can ever realise $\phi(0,0) = A$, so we must conclude that $A$ cannot, in fact, quadratically shatter $(X, Y)$.
\end{proof}

\section{A lower bound on the $\VC_2$-dimension of the quadratic Green-Sanders example}\label{sec:lower}

Considering that the $\VC$-dimension of $\GS(p,n)$ is at most $3$, it is tempting to conjecture that the true $\VC_2$-dimension of $\QGS(p,n)$ is much smaller than $br(5;3,3)$. This section gives a construction which shows that the $\VC_2$-dimension of $\QGS(p,n)$ is at least $3$ for any $p > 2$. The goal is to exhibit a pair of sets $X,Y$ of size $\abs{X}=\abs{Y}=3$ such that $A = \QGS(p,n)$ quadratically shatters $(X,Y)$.

Given a sequence $i_1, \ldots, i_k \in [n]$, let $Q_{i_1, \ldots, i_k}$ denote the joint function $(Q_{i_1}, \ldots, Q_{i_k})$ taking values in $\vecsp{p}{k}$. Note that if $x \in \vecsp{p}{n}$ and $Q_{1, \ldots, k}(x)$ is non-zero for some $k < n$, then $Q_{1, \ldots, k}(x)$ completely determines whether $x \in \QGS(p,n)$. This simple observation suggests that it may be sufficient to consider only a small subset of $\set{Q_i}$ in order to quadratically shatter a pair of sets.

Additionally, in order to have control over whether an element of $\vecsp{p}{n}$ is in $A$, it is useful to consider \emph{quadratic factors}, the type of partition produced by the quadratic arithmetic regularity lemma.

\begin{definition}[Quadratic Factor]
\label{def:poly-factor}
A \emph{quadratic factor $\mathcal{B}$} of complexity $D$ is a partition of $\vecsp{p}{n}$ into the level sets of some polynomials $P_1, \ldots, P_D: \vecsp{p}{n} \rightarrow \vecsp{p}{}$ of degree at most $2$. In other words, the parts of $\mathcal{B}$ have the form $B = \set{x \in \vecsp{p}{n} \st (P_1(x), \ldots, P_D(x))=\vecline{c}}$, where $\vecline{c} \in \vecsp{p}{D}$. Such a part is called an \emph{atom of $\mathcal{B}$} and $\vecline{c}$ is referred to as its \emph{label}.
\end{definition}

By definition, an atom $B$ with label $\vecline{c}$ is non-empty if and only if there is an element $x \in \vecsp{p}{n}$ that simultaneously satisfies $P_1(x) = c_1, \ldots, P_D(x) = c_D$. It turns out that when the quadratic polynomials inducing $\mathcal{B}$ are defined by matrices of sufficiently high rank, we can ensure that all atoms are non-empty via the following lemma \cite[Lemma 4.2]{montreal}.

\begin{lemma}[Size of atoms]
\label{lemma:quadratic-atom-size}
    Let $\mathcal{B}$ be a quadratic factor defined by polynomials $L_1, \ldots, L_l, P_1, \ldots, P_q$, where $\set{L_i}$ are linearly independent linear polynomials, and $\set{P_i}$ are quadratic polynomials of the form $x^T M_i x$ for some matrix $M_i$. Then for any $\vecline{c} \in \vecsp{p}{D}$, the atom $B$ with label $\vecline{c}$ satisfies $\babs{\abs{B} - p^{-(l+q)}\abs{G}} \leq p^{-r/2}\abs{G}$, where
    $r = \min({\rank(\lambda_1 M_1 + \ldots + \lambda_q M_q) \st \underline{\lambda} \in \vecsp{p}{q}\backslash\set{0}}).$
\end{lemma}
In particular, if the quadratic polynomials in Lemma \ref{lemma:quadratic-atom-size} are taken to be a subset of $Q_1, \ldots, Q_n$, then $r = n$ so all atoms of $\mathcal{B}$ are non-empty as long as the complexity of $\mathcal{B}$ is less than $n/2$.

While the goal of this section is to prove a lower bound of $3$, the general approach will be demonstrated by first obtaining a weaker bound of $2$. The main steps of the proof below are labelled and highlighted in bold to serve as a blueprint for the proof of the eventual stronger result.
\begin{theorem}
\label{theorem:2bound}
For all primes $p>2$ and all $n \geq 13$, $\QGS(p,n)$ has $\VC_2$-dimension at least $2$.
\end{theorem}
\begin{proof}
We are looking to find $x, y \in \vecsp{p}{n}$ such that the pair of sets $(X, Y)$ given by $X = \set{0, x}$ and $Y = \set{0, y}$ is quadratically shattered by $A$.

\textbf{Picking $x$ and $y$.}
Let $x$ be any non-zero vector in $\vecsp{p}{n}$. Observe that the vectors $M_1 x$ and $M_2 x$ are linearly independent, since any non-zero linear combination of $M_1$ and $M_2$ has rank $n$ and is therefore an invertible matrix.

Let $H_x = \gen{M_1x, M_2x}^\perp$. This is a subspace of codimension 2 and so is any set of the form $M H_x = \set{Mv: v \in H_x}$, where $M$ is an invertible matrix. Then
$$H^* = M_1^{-1} H_x \cap M_2^{-1} H_x \cap H_x = \gen{x, M_1^{-1}M_2x, M_2^{-1}M_1x, M_1x, M_2x}^\perp$$
is a subspace of codimension at most 5. This means that for any $n \geq 6$, there exists a non-zero $y \in H^*$, orthogonal to each of the vectors on the right-hand side above.

\textbf{Defining a quadratic factor.}
Note that the choice of $x$ and $y$ above ensures that the set $L = \set{2M_1 x, 2M_2 x, 2M_1 y, 2M_2 y}$ is linearly independent. Indeed, if
$$\lambda_1 M_1 x + \lambda_2 M_2 x = \lambda_3 M_1 y + \lambda_4 M_2 y$$
for some $\lambda_1, \ldots, \lambda_4 \in \vecsp{p}{}$, then each side of the equation must be zero because the two sides are orthogonal to each other. On the other hand, both $\set{M_1x, M_2x}$ and $\set{M_1y, M_2y}$ are linearly independent sets, so $\lambda_i = 0$ for all $i \in [4]$.

As a consequence, the quadratic factor $\mathcal{B}$ defined by $L \cup \set{Q_1, Q_2}$ has complexity $6$ and all of its atoms are non-empty by Lemma \ref{lemma:quadratic-atom-size} and the fact that $\mathcal{B}$ has rank $n \geq 13$. Additionally, $y \in H_x$ implies that $x^TM_1y = (M_1x)^Ty = 0$ and $x^TM_2y=0$.

\textbf{Shattering $(X, Y)$.} Let $\phi$ be any containment map on $[0,1]^2$. The goal is to exhibit a $z \in \vecsp{p}{n}$ such that $z \in \phi(0, 0)$, $x+z \in \phi(1, 0)$, $y+z \in \phi(0, 1)$, and $x+y+z \in \phi(1,1)$. This can be achieved by picking $z$ to lie in a specific atom of $\mathcal{B}$, as follows. Observe that
\begin{align}
    \label{l1}
    Q_{i}(x+z) &= Q_{i}(x) + Q_{i}(z) + 2(M_i x)^T z\\
    \label{l2}
    Q_{i}(y+z) &= Q_{i}(y) + Q_{i}(z) + 2(M_i y)^T z\\
    \label{l3}
    Q_{1,2}(x+y+z) &= Q_{1,2}(x+z) + Q_{1,2}(y+z) - Q_{1,2}(z)
\end{align}
for any $i \in \set{0,1}$. In particular, the value of $Q_{1,2}(x+y+z)$ depends solely on those of $Q_{1,2}(z)$, $Q_{1,2}(x+z)$, and $Q_{1,2}(y+z)$. While this may appear restrictive, note that $z$ can be picked in such a way that the values of $Q_{1,2}(x+z)$ and $Q_{1,2}(y+z)$ are independent of $Q_{1,2}(x), Q_{1,2}(y)$ and $Q_{1,2}(z)$ due to the cross-terms $(M_i x)^T z$ and $(M_i y)^T z$ in the corresponding equations. Specifically, given any $a_i, b_i, q_i \in \vecsp{p}{}$, it is possible to ensure that $Q_{1,2}(z) = (q_1,q_2)$, $Q_{1,2}(x+z) = (a_1,a_2)$ and $Q_{1,2}(y+z) = (b_1,b_2)$ by picking $z$ from the atom of $\mathcal{B}$ with label $(a_1 - q_1 - Q_1(x), a_2 - q_2 - Q_2(x), b_1 - q_1 - Q_1(y), b_2 - q_2 - Q_2(y), q_1, q_2).$ This achieves the desired assignment as may be verified by substitution into \eqref{l1} and \eqref{l2}.

Now all that remains is to make sure that $q_i$, $a_i$ and $b_i$ can always be picked in such a way that $\phi(1, 1)$ is realised. The table below shows one way of doing so, with $*$ standing for a `wildcard' value, where any choice is permissible.
\begin{center}
\begin{tabular}{|c|c|c||c|}
    \hline
    $Q_{1,2}(x+z)$ & $Q_{1,2}(y+z)$ & $Q_{1,2}(z)$ & $Q_{1,2}(x+y+z)$\\
    \hline\hline
    $1*$ & $0*$ & $0*$ & $1*$\\
    \hline
    $2*$ & $0*$ & $0*$ & $2*$\\
    \hline
    $0*$ & $1*$ & $0*$ & $1*$\\
    \hline
    $0*$ & $2*$ & $0*$ & $2*$\\
    \hline
    $2*$ & $-1*$ & $0*$ & $1*$\\
    \hline
    $1*$ & $1*$ & $0*$ & $2*$\\
    \hline
\end{tabular}
\end{center}
Here the first two rows correspond to the case when $x+y+z \in A$ if and only if $x+z \in A$; the next two rows apply when $x+y+z \in A$ if and only if $y+z \in A$; finally, the last two rows cover the case of $x+y+z$ having the opposite containment to both $x+z$ and $y+z$. This, in fact, covers all possible cases, and thus $(X, Y)$ can indeed be quadratically shattered by $A$. Hence the $\VC_2$-dimension of $A$ is at least $2$.
\end{proof}
To obtain a lower bound of $3$ on the $\VC_2$-dimension of $\QGS(p,n)$, the above argument can be extended by using a quadratic factor of higher complexity. Specifically, we will work with $Q_{1,2,3}$ rather than just $Q_{1,2}$.

\begin{proof}[Proof of Theorem \ref{theorem:lower}]
    Proceeding as before, it is possible to select $X = \set{0, x_1, x_2}$ and $Y = \set{0, y_1, y_2}$, where $x_1, x_2, y_1, y_2 \in \vecsp{p}{n}\backslash \set{0}$, such that the set
        $$L = \set{2 M_i x_j \st 1 \leq i \leq 3, 1 \leq j \leq 2} \cup \set{2 M_i y_j \st 1 \leq i \leq 3, 1 \leq j \leq 2}$$
    is linearly independent, and for each $i \in [3]$ and $j_1, j_2 \in [2]$, $x_{j_1}^T M_i y_{j_2} = 0$. Then the quadratic factor $\mathcal{B}$ defined by $L \cup \set{Q_1, \ldots, Q_3}$ has complexity 15, and all of its atoms are non-empty by Lemma \ref{lemma:quadratic-atom-size} and the fact that $\mathcal{B}$ has rank $n \geq 31$.

    Specifically, start by picking any linearly independent $x_1, x_2 \in \vecsp{p}{n}$ and noting that, for any $i \in [n]$, $M_i x_1$ and $M_i x_2$ are also linearly independent. Define the subspace $H_x = \gen{M_i x_j \st 1 \leq i \leq 3, 1 \leq j \leq 2}^\perp$ and pick two linearly independent $y_1, y_2$ from
    $$H_x \cap \bigcap_{i=1}^3 M^{-1}_i H_x = \gen{x_j, M_{i_1}^{-1} M_{i_2} x_j, M_{i_1}x_j \st 1 \leq i_1 \neq i_2 \leq 3, j = 1,2}^\perp,$$
    which is itself a subspace of codimension at most 20.

    As in the proof of Theorem \ref{theorem:2bound}, $\set{x_1, x_2}$ and $\set{y_1, y_2}$ chosen in this way have the property that for any $a_1, a_2, b_1, b_2, q \in \vecsp{p}{3}$, it is always possible to find a $z \in \vecsp{p}{n}$ such that $Q_{1,2,3}(x_i+z) = a_i$, $Q_{1,2,3}(y_j+z) = b_j$, and $Q_{1,2,3}(z) = q$ by picking $z$ from an atom of $\mathcal{B}$ with an appropriate label. Therefore, in order to quadratically shatter $(X, Y)$ by $A$, it is sufficient to exhibit, for each containment map $\phi$, values of $a_1, a_2, b_1, b_2, q \in \vecsp{p}{3}$ that realise it. While it is no longer practical to write down all necessary configurations in a single table, the number of separate cases to consider may be reduced to the following.
    
    \textbf{Case 1:} $\phi(1,0) = \phi(1, 1) = \phi(1, 2)$.
    Set $Q_1(y_1+z), Q_1(y_2+z), Q_1(x_2+z)$ and $Q_1(z)$ to zero
    and let $Q_1(x_1+z)$ be equal to either $1$ or $-1$, depending on $\phi(1, 0)$. Since
    $$Q_{1,2,3}(x_i+y_j+z) = Q_{1,2,3}(x_i+z) + Q_{1,2,3}(y_j+z) - Q_{1,2,3}(z),$$
    this immediately realises $\phi(1, 0)$, $\phi(1,1)$ and $\phi(1,2)$, regardless of the value of $Q_{2,3}$ at $x_1+z$, $x_1+y_1+z$ and $x_1 + y_2 +z$. To realise $\phi(2,0)$ and $\phi(0,0)$, select appropriate values from the following table.
    \begin{center}
    \begin{tabular}{|c|c||c|}
        \hline
        $Q_{2,3}(x_2+z)$ & $Q_{2,3}(z)$ & $Q_{2,3}(x_2+z)-Q_{2,3}(z)$\\
        \hline\hline
        $22$ & $11$ & $11$\\
        \hline
        $12$ & $01$ & $11$\\
        \hline
        $10$ & $0{-1}$ & $11$\\
        \hline
        $02$ & $-11$ & $11$\\
        \hline
    \end{tabular}
    \end{center}
    Finally, the rest of $\phi$ can be realised by choosing $Q_{2,3}(y_j+z)$ to take the value $0*$, $10$ or $-11$, as appropriate.

    \textbf{Case 2:} $\phi(1, 0) \neq \phi(1, 1) = \phi(1, 2)$, $\phi(2, 1) = \phi(2, 2)$.    
    Let $Q_{1,2,3}(y_j+z) = 001$ if $\phi(0,j)=A$, and $Q_{1,2,3}(y_j+z) = 002$ otherwise, for each $j \in [2]$. Then select $Q_{1,2,3}(z)$ and $Q_{1,2,3}(x_i+z)$ from one of the tables below, depending on the value of $\phi(0,0)$.
    \begin{table}[h!]
    \begin{subtable}[h]{0.49\textwidth}
        \centering
        \begin{tabular}{|c||c|}
            \hline
            $Q_{1,2}(x_i+z)$ & $Q_{1,2}(x_i+y_j+z)$\\
            \hline\hline
            $2*$ & $1*$\\
            \hline
            $11$ & $01$\\
            \hline
            $0*$ & ${-1}*$\\
            \hline
        \end{tabular}
        \caption{$Q_{1,2}(z) = 10$.}
    \end{subtable}
    \hfill
    \begin{subtable}[h]{0.49\textwidth}
        \centering
        \begin{tabular}{|c||c|}
            \hline
            $Q_{1,2}(x_i+z)$ & $Q_{1,2}(x_i+y_j+z)$\\
            \hline\hline
            $-1*$ & $0*$\\
            \hline
            $1*$ & $2*$\\
            \hline
            $01$ & $11$\\
            \hline
        \end{tabular}
        \caption{$Q_{1,2}(z) = -10$.}
    \end{subtable}
    \end{table}
    
    \textbf{Case 3:} $\phi(1, 0) \neq \phi(1, 1) = \phi(1, 2) = \phi(2,1) \neq \phi(2, 2)$. We may additionally assume that $\phi(0, 1) = \phi(1, 0)$, as otherwise $\phi(0, 1) = \phi(1, 1) = \phi(2, 1)$ so Case 1 would apply by symmetry.
    
    \textbf{Subcase 3.1: $\phi(2, 0) = \phi(0,0)$.} Set $Q_{1}(y_1+z)=Q_1(y_2+z)=0$, and $Q_{1,2,3}(x_2+z)=Q_{1,2,3}(z)= \pm 100$, according to $\phi(2, 0)$. This ensures that for each $j \in [2]$,
    $Q_1(x_2+y_j+z)=Q_1(x_2+z) + Q_1(y_j+z) - Q_1(z) = 0.$
    
    Now pick an appropriate value for $Q_{1,2}(x_1+z)$ from the table appearing in Case 2 in order to realise $\phi(1,1)$ and $\phi(1,2)$. This is possible because, by assumption, $\phi(1,0) \neq \phi(1, j)$ for each $j \in [2]$, and the rows of the table that satisfy this condition do not depend on the value of $Q_2(y_j)$. Finally, choose $Q_{2,3}(x_2+z)$ and $Q_{2,3}(y_j+z)$ from the following tables to realise $\phi(0,j)$ and $\phi(2,j)$.
    \begin{table}[h!]
    \begin{subtable}[h]{0.49\textwidth}
        \centering
        \begin{tabular}{|c||c|}
            \hline
            $Q_{2,3}(y_j+z)$ & $Q_{2,3}(x_2+y_j+z)$\\
            \hline\hline
            $10$ & $20$\\
            \hline
            $0*$ & $1*$\\
            \hline
            ${-1}2$ & $02$\\
            \hline
        \end{tabular}
        \caption{$Q_{2,3}(x_2+z) = 10$.}
    \end{subtable}
    \hfill
    \begin{subtable}[h]{0.49\textwidth}
        \centering
        \begin{tabular}{|c||c|}
            \hline
            $Q_{2,3}(y_j+z)$ & $Q_{2,3}(x_2+y_j+z)$\\
            \hline\hline
            $20$ & $10$\\
            \hline
            $0*$ & ${-1}*$\\
            \hline
            $11$ & $01$\\
            \hline
        \end{tabular}
        \caption{$Q_{2,3}(x_2+z) = -10$.}
    \end{subtable}
    \end{table}

    \textbf{Subcase 3.2: $\phi(0,1)=\phi(0,2).$} Set $Q_{1}(x_2+z)=\pm 100$, and choose $Q_{1,2,3}(y_j+z)$ and $Q_{1,2,3}(z)$ from the table. Note that any such choice ensures that $Q_{1}(x_2+y_j+z)=0$.
    \begin{center}
    \begin{tabular}{|c|c||c|}
        \hline
        $Q_{1,2,3}(y_j+z)$ & $Q_{1,2,3}(z)$ & $Q_{1,2,3}(y_j+z)-Q_{1,2,3}(z)$\\
        \hline\hline
        $10*$ & $200$ & $-10*$\\
        \hline
        $01*$ & $110$ & $-10*$\\
        \hline
        $02*$ & $-120$ & $10*$\\
        \hline
        $20*$ & $100$ & $10*$\\
        \hline
    \end{tabular}
    \end{center}
    The value of $*$ in the table may be chosen independently for $y_1$ and $y_2$, in a way that realises $\phi(2,1)$ and $\phi(2,2)$. Finally, set $Q_{1,2,3}(x_1+z)$ to be $00*$, $200$ or $100$ to realise the rest of $\phi$.

    \textbf{Subcase 3.3: $\phi(0,1) \neq \phi(0,2)$, $\phi(2,0) \neq \phi(0,0)$.} We may assume that $\phi(0,2) \neq \phi(0,0)$ and $\phi(1,0) \neq \phi(2,0)$, since otherwise Subcase 3.1 or Subcase 3.2 would apply by symmetry. Then $\phi(1,0)=\phi(0,1)=\phi(0,0) \neq \phi(2, 0) = \phi(0, 2).$ Combined wtih the assumptions on the whole case, this results in
    $\phi(0,0) = \phi(2,2) = \phi(1,0) = \phi(0,1) \neq \phi(0, 2) = \phi(2, 0) = \phi(1, 2) = \phi(2, 1) = \phi(1, 1).$
    
    Set $Q_{1,2,3}(z) = 00*$ and $Q_1(x_1+z)=Q_1(y_1+z)=0$. The rest of $\phi$ is now realised by appropriate choices from the following tables, depending on the value of $\phi(2,0)$.
    \vspace{-0.2cm}
    \begin{table}[h!]
    \begin{subtable}[h]{0.49\textwidth}
    \centering
    \begin{tabular}{|c|c|}
        \hline
        $Q_{1,2,3}(x_2+z)$ & $Q_{1,2,3}(y_2+z)$\\
        \hline\hline
        $100$ & $100$\\
        \hline
        $200$ & $-100$\\
        \hline
    \end{tabular}
    \end{subtable}
    \hfill
    \begin{subtable}[h]{0.49\textwidth}
    \centering
    \begin{tabular}{|c|c|}
        \hline
        $Q_{2,3}(x_1+z)$ & $Q_{2,3}(y_1+z)$\\
        \hline\hline
        $10$ & $10$\\
        \hline
        $20$ & $-10$\\
        \hline
    \end{tabular}
    \end{subtable}
    \end{table}

    Note that, by swapping the roles of $y_1$ and $y_2$, Case 3 also realises any $\phi$ satisfying $\phi(1, 0) \neq \phi(1, 1) = \phi(1, 2) = \phi(2,2) \neq \phi(2, 1)$. As such, taken together, Cases 1-3 cover any scenario in which $\phi(1, 1) = \phi(1, 2)$. By symmetry, they also apply to $\phi(1, 1) = \phi(2, 1)$, $\phi(2, 2) = \phi(2, 1)$, or $\phi(2, 2) = \phi(1, 2)$. This leaves the following as the only remaining possibility.
    \vspace{0.1cm}
    
    \textbf{Case 4:} $\phi(2, 2) = \phi(1, 1) \neq \phi(1, 2) = \phi(2, 1)$. Without loss of generality, assume that $\phi(2, 2) = \phi(1, 1) = A$ and $\phi(1, 2) = \phi(2, 1) = A^C$.

    \textbf{Subcase 4.1: $\phi(0, 1) \neq \phi(0, 2)$.} One of $\phi(0, 1)$ or $\phi(0, 2)$ must be equal to $\phi(1, 1) = \phi(2, 2) = A$. Swapping the roles of $y_1$ and $y_2$ if necessary, we may assume $\phi(0, 1) = A$ so that $\phi(0, 2) = A^C$. Let $Q_{1,2,3}(y_1+z) = 120$, and set $Q_{1,2, 3}(z)$ and $Q_{1,2,3}(x_1+z)$ to $00*$, with an appropriate value of $*$ for each. This already realises $\phi(1,1)$ and $\phi(1,2)$ for any non-zero choice of $Q_{1,2,3}(y_2+z)$. For the rest of $\phi$, choose the values of $Q_{1,2,3}$ at $x_2+z$ and $y_2+z$ according to the table.

    \begin{center}
    \begin{tabular}{|c|c||c|c|}
        \hline
        $Q_{1,2,3}(x_2+z)$ & $Q_{1,2,3}(y_2+z)$ & $Q_{1,2,3}(x_2+y_1+z)$ & $Q_{1,2,3}(x_2+y_2+z)$\\
        \hline\hline
        $100$ & $-110$ & $22*$ & $01*$\\
        \hline
        $-100$ & $200$ & $02*$ & $10*$\\
        \hline
    \end{tabular}
    \end{center}
    By symmetry, this subcase also applies when $\phi(1,0) \neq \phi(2,0)$.

    \textbf{Subcase 4.2: $\phi(0, 1) = \phi(0, 2) = \phi(1, 0) = \phi(2, 0)$.} Use one of the tables below, in which the cell with label $(i,j) \in [0,2]^2$ contains the value of $Q_{1,2,3}(x_i+y_j+z)$, with $x_0=y_0=0$.
\begin{table}[h!]
    \begin{subtable}[h]{0.45\textwidth}
        \centering
        \begin{tabular}{|c|c|c|c|}
            \hline
            {} & \textbf{0} & \textbf{1} & \textbf{2}\\
            \hline
            \textbf{0} & $00*$ & $100$ & $010$\\
            \hline
            \textbf{1} & $010$ & $11*$ & $02*$\\
            \hline
            \textbf{2} & $100$ & $20*$ & $11*$\\
            \hline
        \end{tabular}
        \caption{$\phi(0, 1) = A$}
    \end{subtable}
    \hfill
    \begin{subtable}[h]{0.45\textwidth}
    \centering
    \begin{tabular}{|c|c|c|c|}
        \hline
        {} & \textbf{0} & \textbf{1} & \textbf{2}\\
        \hline
    \textbf{0} & $00*$ & $020$ & $200$\\
        \hline
        \textbf{1} & $0{-1}0$ & $01*$ & $2{-1}*$\\
        \hline
        \textbf{2} & $-100$ & $-12*$ & $10*$\\
        \hline
    \end{tabular}
    \caption{$\phi(0, 1) = A^C$}
    \end{subtable}
    \end{table}

    \textbf{Subcase 4.2: $\phi(0, 1) = \phi(0, 2) \neq \phi(1, 0) = \phi(2, 0)$.} Without loss of generality, assume that $\phi(0, 1) = \phi(0, 2) = A$, then use one of the tables below depending on the value of $\phi(0,0)$.
    
    \begin{table}[h!]
    \begin{subtable}[h]{0.45\textwidth}
    \centering
    \begin{tabular}{|c|c|c|c|}
        \hline
        {} & \textbf{0} & \textbf{1} & \textbf{2}\\
        \hline
    \textbf{0} & $001$ & $100$ & $110$\\
        \hline
        \textbf{1} & $-111$ & $010$ & $020$\\
        \hline
        \textbf{2} & $-100$ & $00{-1}$ & $01{-1}$\\
        \hline
    \end{tabular}
    \caption{$\phi(0,0) = A$}
    \label{subtable:A}
    \end{subtable}
    \hfill
    \begin{subtable}[h]{0.45\textwidth}
    \centering
    \begin{tabular}{|c|c|c|c|}
        \hline
        {} & \textbf{0} & \textbf{1} & \textbf{2}\\
        \hline
    \textbf{0} & $00{-1}$ & $100$ & $110$\\
        \hline
        \textbf{1} & $-110$ & $011$ & $021$\\
        \hline
        \textbf{2} & $-101$ & $002$ & $012$\\
        \hline
    \end{tabular}
    \caption{$\phi(0,0) = A^C$}
    \label{subtable:notA}
    \end{subtable}
    \end{table}

This completes the case analysis, which shows that $(X, Y)$ defined at the start of the proof can indeed be shattered by $\QGS(p,n)$, and therefore the $\vctwo$-dimension of $\QGS(p,n)$ is at least $3$.
\end{proof}

As it stands, this construction cannot be extended to obtain a lower bound of $4$: if we choose $X$ and $Y$ so that all cross-terms have the same value (taken to be zero throughout this section), then the proof of Theorem \ref{theorem:upper} implies that $(X, Y)$ is \textit{not} quadratically shattered by $A$. It may still be possible to construct suitable $X$ and $Y$ using the method outlined here, but the cross-terms $\mu_{i,j}^\bracketed{t}$ would have to have different values. On the other hand, the $\VC_2$-dimension of $\QGS(p,n)$ may be exactly $3$. Whether or not this is the case remains an interesting open question.

\bibliography{references.bib}
\bibliographystyle{amsplain}

\appendix
\section{Multicolour bipartite Ramsey numbers}
\label{appendix:bipartite-ramsey}

In \cite{bipartite-ramsey}, Conlon proved that $br(2;t,t) \leq (1 + o(1))2^{t+1}\log(t)$. The same proof shows that for any $r \geq 2$, $br(r;t,t) \leq (1 + o(1))r^{t+1}\log(t)$ (the only modification needed is to replace the subgraph density of $\frac12$ by $\frac1{r}$ throughout the argument). The proof relies on a lemma which, given an edge density $\rho$, guarantees a copy of $K_{q,s}$ in any large enough bipartite graph of density $\rho$. The version stated here is easily extracted from the proof of \cite[Lemma 1]{bipartite-ramsey} before an asymptotic bound is applied.

\begin{lemma}
\label{lemma:density-bipartite}
    Let $G$ be a subgraph of $K_{m,n}$ with edge density $\rho$, where $0 < \rho \leq 1$. Then, provided that $m > (q-1)\rho^{-1}$ and $$ n > \frac{\binom{m}{q}}{\binom{\rho m}{q}} (s-1),$$
    $G$ must contain a copy of $K_{q,s}$.
\end{lemma}

\noindent The bound on $m$ is necessary to ensure that the function $F(x) = \binom{x}{q}$ is convex so as to allow an application of Jensen's inequality in the proof. We now verify that in the case $q=s=3$, Lemma \ref{lemma:density-bipartite} guarantees the following.

\begin{prop}
\label{prop:Ramseybound}
    For any $r \geq 2$, $\text{br}(r; 3, 3) \leq 4r^3+1$.
    %\cmnt{$3k^3+k^2+1$ might work also}
\end{prop}
\begin{proof}
    Fix an $r$-colouring of $K_{n, n}$, where $n = 4r^3+1$. Denote the two vertex classes of $K_{n, n}$ by $X$ and $Y$, respectively. For each vertex $v \in X$, there is a colour that appears on at least $4{r}^2+1$  edges incident with $v$. Call such a colour `dominant'. By averaging, there is a colour $c \in [r]$ which is dominant for $m \geq 4r^2+1$ vertices in $X$. Let $X'$ denote the set of these vertices, and let $H$ be the graph on $X' \cup Y$ consisting only of the edges coloured $c$. Note that $H$ is a subgraph of $K_{m, n}$.
    
    The average degree of the vertices of $Y$ in $H$ is
    $$\frac{(4r^2+1)^2}{4r^3+1} = \frac{16r^4+8r^2+1}{4r^3+1} = 4r+\frac{8r^2-4r+1}{4r^3+1}.$$
    In particular, there must be a vertex $y \in Y$ of degree at least $q = 4r+1$ and, therefore, there is a copy of $K_{q, 1}$ in $H$.

    Let $X''\subseteq X'$ denote the $q$ vertices of this instance of $K_{q, 1}$, and $Y' = Y \backslash \set{y}$. Let $H'$ be the induced subgraph of $H$ on $(X'', Y')$. The total number of edges in $H'$ is at least $4r^2 q$, which means its edge density $\rho'$ in $K_{q, n-1}$ satisfies $\rho' \geq {4r^2}/(n-1) = 1/r.$
    
    By Lemma \ref{lemma:density-bipartite}, $H$ has a copy of $K_{3,2}$ so long as $q > 2\rho'^{-1}$ (which indeed holds since $q > 4r$) and additionally
        $$n-1 > \frac{\binom{q}{3}}{\binom{\rho' q}{3}}$$
    To see that the latter holds, observe that $\rho' q > 4$, and therefore
    \begin{align*}
        \frac{\binom{q}{3}}{\binom{\rho' q}{3}} &< \frac{q(q-1)(q-2)}{4 \times 3 \times 2} = \frac{4r(16r^2-1)}{24} < 4r^3 = n-1.
    \end{align*}
    Hence, $H'$ contains a copy of $K_{3,2}$.
    
    On the other hand, all vertices of $X''$ were part of a copy of $K_{q, 1}$ of the same colour, so adding $y$ to $K_{3,2}$ gives a monochromatic $K_{3,3}$, as required.
\end{proof}

\end{document}